\documentclass[11pt]{amsart}
\usepackage{amsmath,amsthm, amscd, amssymb, amsfonts}

\usepackage[]{color} 

\usepackage{graphicx} 

\usepackage{mathrsfs} 
\usepackage{figsize}
\usepackage{algorithm}
\usepackage{algpseudocode}
\usepackage{enumerate} 

\usepackage{float} 

\usepackage{tikz} 

\usepackage{tkz-berge} 
\usepackage{tkz-graph}
\usetikzlibrary{arrows}
\usetikzlibrary{calc}
\usetikzlibrary{matrix}
\usetikzlibrary{decorations.pathreplacing}  

\usepgflibrary{shapes.geometric} 
\usetikzlibrary[topaths]
\usetikzlibrary{positioning}
\usetikzlibrary{arrows}  
\usetikzlibrary{graphs,graphs.standard}

\usetikzlibrary{arrows.meta,calc}

\usepackage{tabularx}

\theoremstyle{plain}
\newtheorem{theorem}{Theorem}[section]
\newtheorem{corollary}[theorem]{Corollary}
\newtheorem{proposition}[theorem]{Proposition}
\newtheorem{lemma}[theorem]{Lemma}
\newtheorem{definition}{Definition}[section]

\newtheorem{example}[theorem]{Example}

\usepackage[colorlinks]{hyperref}

\def\F{\mathbb{F}}

 \def\ad{\mathrm{ad}}

\title[Comaximal Graphs of finite-dimensional Lie algebras]{Comaximal Graphs of finite-dimensional Lie algebras over finite fields: Triangle counts and structural invariants}

\thanks{...}

\author{David A. Towers}
\address{Lancaster University, School of Mathematical Sciences, Lancaster, LA1 4YF, UK} 
  \email{d.towers@lancaster.ac.uk}

\author{Yesneri Zuleta}
 \address{Universidad del Norte,  Departamento de Matemáticas, Física y Ciencia de Datos, Km 5 vía a Puerto Colombia, Barranquilla, Colombia.}
  \email{zuletay@uninorte.edu.co}

\author{Ismael Gutierrez}
 \address{Universidad del Norte,  Departamento de Matemáticas, Física y Ciencia de Datos, Km 5 vía a Puerto Colombia, Barranquilla, Colombia.}
  \email{isgutier@uninorte.edu.co }


\begin{document}

\begin{abstract}
Let $L$ be a finite-dimensional Lie algebra over a field $F$. The comaximal graph $\Gamma(L)$ has as vertices the proper nonzero subalgebras of $L$, two of them adjacent whenever they generate $L$; its structure was previously classified for Lie algebras of dimension at most 3 over finite fields. Here we extend that work in two directions. First, we obtain explicit formulas for the number of triangles $t(\Gamma(L))$ for every three-dimensional Lie algebra over $\F_q$. Second, we extend the classification to several four-dimensional families over $\mathbb{F}_q$, the abelian, Heisenberg, and filiform algebras, and $\mathfrak{gl}_2(\F_q)$. We also relate graph-theoretic properties of $\Gamma(L)$, such as completeness and the role of the Frattini subalgebra, to structural properties of $L$, including supersolvability. These results yield new combinatorial invariants for finite-dimensional Lie algebras over finite fields.
\end{abstract}

\maketitle


\section{Introduction}

Graphs associated with algebraic structures provide a fruitful way of translating algebraic information into combinatorial language. Classic examples include commuting graphs, non-commuting graphs, power graphs, intersection graphs, and various graphs attached to groups, rings, and vector spaces. Such constructions often reveal structural properties that are not immediately visible from an algebraic viewpoint and allow the use of graph-theoretic methods in algebra.

More recently, a systematic study of graphs associated with finite-dimensional
Lie algebras were introduced in \cite{TGF, TGF2, TZG}. In \cite{TGF} the nilpotent graph of a Lie algebra was introduced, whose vertices are the non-nilpotent elements of $L$, with two vertices adjacent if and only if they generate a nilpotent subalgebra. In \cite{TGF2} an analogous construction is carried out for solvability, defining the solvable graph of $L$.  These investigations established several structural and combinatorial properties of the corresponding graphs, including connectedness, diameter, clique number, and classifications over finite
fields. The \emph{comaximal graph} of a finite-dimensional Lie algebra was introduced in \cite{TZG}, which is the focus of the present paper.

The \emph{comaximal graph} of $L$, denoted by $\Gamma(L)$, is the simple graph whose vertices are the non-trivial proper Lie subalgebras of $L$, where two
distinct vertices $A$ and $B$ are adjacent whenever $\langle A, B\rangle = L$.
Thus, adjacency reflects the ability of two proper subalgebras to generate
the whole algebra. This construction is the natural analog of the comaximal graph of a ring for Lie algebras.

Unlike the nilpotent and solvable graphs, whose vertices are individual
elements of $L$, the vertices of $\Gamma(L)$ are themselves subalgebras, so
that $\Gamma(L)$ records how the subalgebra lattice of $L$ can be combined
to recover the whole algebra. In \cite{TZG}, the basic structural theory
of $\Gamma(L)$ was laid out. Explicit descriptions of the vertices, adjacency relations, connected components, and several graph-theoretic invariants were obtained. These results naturally lead to a more detailed investigation of the
local combinatorial structure of the graph, together with a full classification of $\Gamma(L)$ for Lie algebras of dimension at most three over a finite field $\F_q$.

The present paper continues this line of investigation, pushing the classification of \cite{TZG} further in two complementary directions. First, building on the three-dimensional classification already in hand, we study a finer combinatorial invariant of $\Gamma(L)$: the number of triangles $\mathfrak{t}(\Gamma(L))$, which we compute explicitly for every three-dimensional Lie algebra over $\F_q$, using the classification obtained in \cite{TZG}, and organized according to the dimension of the derived algebra $L'$. Triangle counts measure local clustering and provide information about the distribution of complete subgraphs. In graphs arising from algebraic structures, triangles frequently encode significant algebraic information and often distinguish between objects with similar global parameters. Second, we extend the classification itself beyond dimension three, describing the comaximal graph for several families of four-dimensional Lie algebras over $\F_q$: the abelian, Heisenberg-plus, and filiform nilpotent algebras, as well as
$\mathfrak{gl}_2(\F_q)$, and relating the structure of $\Gamma(L)$ to algebraic properties of $L$ such as the Frattini subalgebra and supersolvability. Together, these results illustrate how the comaximal graph encodes the internal subalgebra structure of $L$ with increasing fidelity as one moves to finer invariants and to
more intricate algebras.

A second objective is to investigate how structural properties of $\Gamma(L)$ reflect algebraic properties of the Lie algebra $L$. In particular, we study the influence of the Frattini subalgebra, supersolvability, and simplicity on the graph's structure and establish several characterizations in terms of graph-theoretic properties. We also begin the study of the comaximal graph for certain classes of four-dimensional Lie algebras over finite fields, thereby extending the
low-dimensional theory developed previously.

Throughout the paper, finite fields play a central role. Over a finite field $\F_q$, the number of subalgebras is finite and often explicitly computable, allowing precise enumeration of vertices, edges, and triangles. Moreover, the lattice of subalgebras often admits a projective-geometric interpretation, making combinatorial counting arguments particularly effective.

The paper is organized as follows. Section 2 collects the graph-theoretic and Lie-theoretic preliminaries needed later. Section 3 determines the number of triangles in the comaximal graph of every Lie algebra of dimension at most three over $\F_q$.
Section 4 extends the study of comaximal graphs to several classes of four-dimensional Lie algebras over finite fields. Finally, in Section 5, we investigate connections between the graph $\Gamma(L)$ and structural properties of the Lie algebra $L$, such as the Frattini subalgebra and supersolvability.

\section{Preliminaries}
\subsection*{Basics on Graphs.}
A graph $\Gamma = (V, E)$ consists of a set of vertices $V$ and a set of edges $E\subseteq  \{\{u, v\}\mid u,v \in V, u\neq v\}$. All graphs in this paper are simple and undirected. We recall the following standard terminology; for a comprehensive reference, see \cite{Diestel, West, Bondy}. The \emph{order} and \emph{size} of $\Gamma$ are $|V|$ and $|E|$, respectively. The degree $\deg(v)$ of a vertex $v$ is the number of edges incident to $v$. A graph is regular if all vertices have the same degree. The degree sequence of $\Gamma$ is the list of degrees of all vertices in non-increasing order. A graph is \emph{complete} if every pair of distinct vertices is adjacent; the complete graph on $n$ vertices is denoted by $K_n$.

A path of length $k$ in $\Gamma$ is a sequence of distinct vertices $v_0, v_1,\ldots, v_k$ such that $\{v_i, v_{i+1}\} \in E$ for all $i$. The distance $d(u, v)$ between two vertices is the length of a shortest path connecting them, or $\infty$ if no such path exists. The eccentricity $\operatorname{ecc}(v)$ of a vertex $v$ is $\max\{d(v, u)\mid u\in V\}$. The diameter $\operatorname{diam}(\Gamma)$ is the maximum eccentricity. 
A graph $\Gamma$ is connected if there is a path between every pair of vertices. 
A clique is a set of pairwise adjacent vertices. The clique number $\omega(\Gamma)$ is the size of a largest clique. An independent set is a set of pairwise non-adjacent vertices. 


Throughout the paper, we use $\Gamma(L)^*$ to denote the subgraph obtained
from $\Gamma(L)$ by removing the isolated vertices. More generally, if $S$
is a set of subalgebras of $L$, we write $\Gamma(L)[S]$ for the induced
subgraph of $\Gamma(L)$ on $S$. In particular, when $F(L) \neq L$, and $S=\{A < L\mid 0\neq A \not\subseteq F(L)\}$, then we write
\[\Gamma_F(L) := \Gamma(L)[S]\]
for the subgraph obtained from $\Gamma(L)$ by removing the vertices contained
in $F(L)$. Note that $\Gamma(L)^* \subseteq \Gamma_F(L)$, with equality if
and only if every isolated vertex of $\Gamma(L)$ is contained in $F(L)$.

\subsection*{Basics on Lie Algebras.}
Let $L$ be a finite-dimensional Lie algebra over a field $\mathbb{F}$. A \emph{Lie algebra} is a vector space $L$ over $\mathbb{F}$, equipped with a bilinear operation $[\cdot,\cdot]: L \times L \to L$, called the \emph{Lie bracket}, satisfying the following axioms:
\begin{enumerate}
    \item \textbf{Alternating property:} $[x,x] = 0$ \ $\forall  x \in L$;
    \item \textbf{Jacobi identity:} $[x,[y,z]] + [y,[z,x]] + [z,[x,y]] = 0$ $\forall x, y, z \in L$.
\end{enumerate}
These imply \textbf{anticommutativity}: $[x,y] = -[y,x]$ for all $x,y \in L$.

A Lie algebra is said to be \emph{abelian} if $[x,y] = 0$ for all $x,y \in L$. Every one-dimensional Lie algebra is abelian. The \emph{dimension} of a Lie algebra is its dimension as a vector space over $\mathbb{F}$. A \emph{subalgebra} of $L$ is a subspace closed under the Lie bracket. If $A, B$ are subspaces of $L$, then $[A,B]$ denotes the subspace spanned by all $[a,b]$ with $a \in A$, $b \in B$. The \emph{centraliser} of an element $x\in L$ is $C_L(x)=\{y\in L \mid [x,y]=0\}$.

The Frattini subalgebra $F(L)$ of a Lie algebra $L$ is the intersection of all maximal subalgebras of $L$. It is the largest subalgebra of $L$ with the property that every subalgebra generated by $F(L)$ together with any other proper subalgebra is still proper. If $L$ has no maximal subalgebras, we set $F(L) = L$. 

A subalgebra $B$ of $L$ is called a Borel subalgebra if it is a maximal solvable subalgebra of $L$. In the case $\mathfrak{sl}_2(\mathbb{F}_q)$, all Borel subalgebras are conjugate and two-dimensional. An element $x\in L$ is called nilpotent if $\ad\, x$ is a nilpotent linear map, and semisimple if $\ad\, x$ is a semisimple linear map. In $\mathfrak{sl}_2(\mathbb{F}_q)$, a semisimple element is called split if its characteristic polynomial splits over $\F_q$, and nonsplit otherwise.

 Let $\F$ be a field. Standard examples include: 
\begin{itemize}
\item $\mathfrak{gl}_n(\F)$, the Lie algebra of all $n \times n$ matrices over $\F$ with bracket $[x,y] = xy - yx$;
\item $\mathfrak{sl}_n(\F)\subseteq \mathfrak{gl}_n(\F)$, the subalgebra of traceless $n \times n$ matrices over $\F$, which is simple for $n\geq 2$. In particular, $\mathfrak{sl}_2(\F_q)$, the three-dimensional simple Lie algebra over a finite field, which serves as the central example of this paper, with standard basis $\{x, y, h\}$, where
   \[x =   \begin{pmatrix} 
0 & 1 \\ 0 & 0 
\end{pmatrix},
  \quad
  y =   \begin{pmatrix} 
  0 & 0 \\ 1 & 0 
  \end{pmatrix},
  \quad
  h =  \begin{pmatrix} 
  1 & 0 \\ 0 & -1 
  \end{pmatrix},\]
and Lie brackets $[h,x]=2x,\qquad [h,y]=-2y,\qquad [x,y]=h$. 
\end{itemize}

The \emph{derived series} of $L$ is defined recursively by $L^{(0)} = L$, and $L^{(k+1)} = [L^{(k)}, L^{(k)}]$.
Further, the \emph{lower central series} is given by $L^{1}=L$, and $L^{i+1}=[L,L^{i}]$. The derived algebra of $L$, usually denoted by $L'$ is $L^2 = [L, L]$. It is spanned by all products of the form $[x, y]$ with $x, y\in L$; $L$ is called perfect if $L^2 = L$. The Lie algebra $L$ is called \emph{solvable} if $L^{(r)} = 0$ for some $r$, and \emph{nilpotent} if $L^k = 0$ for some $k$. Every nilpotent Lie algebra is solvable, but the converse is not true.

\section{Triangle structure in the comaximal graph} 

In this section, we study the number of triangles in the comaximal graph
$\Gamma(L)$ for several classes of low-dimensional Lie algebras over the
finite field $\F_q$. Recall that a triangle in a graph is a $3$-cycle, equivalently, a set of
three pairwise adjacent vertices. We denote by $\mathfrak{t}(\Gamma)$ the number of
triangles of a graph $\Gamma$.

\subsection{Number of triangles in the comaximal graph of a 1-dimensional Lie algebra}
If $L$ is a $1$-dimensional Lie algebra over the field $\F_q$, then the comaximal graph $\Gamma(L)$ is the empty graph. Thus, it has no triangles. See Section 4.1 in \cite{TZG}.

\subsection{Number of triangles in the comaximal graph of a 2-dimensional Lie algebra}
If $L$ is a $2$-dimensional Lie algebra over the field $\F_q$, then the comaximal graph $\Gamma(L)$ is the complete graph $K_{q+1}$. See Section 4.2 in \cite{TZG}. Therefore, the number of triangles in $\Gamma(L)$ is $\tbinom{q+1}{3}$.

\subsection{Number of triangles in the comaximal graph of a 3-dimensional Lie algebra}
Throughout this subsection, $L$ denotes a three-dimensional Lie algebra over the finite field $\F_q$. We use the notation introduced in subsection 4.3 in \cite{TZG}. For example, $\mathcal L$ denotes the set of 1-dimensional Lie subalgebras of $L$ (lines), and $\mathcal P$ the set of 2-dimensional Lie subalgebras of $L$ (planes).

\smallskip

\begin{lemma}\label{T3P}
Let $\Gamma$ be a graph in which the induced subgraph on a set $\mathcal{P}$ of vertices is a complete graph $K_{m}$.  Then the number of triangles of $\Gamma(L)$ whose three vertices all lie in $\mathcal{P}$ is $\tbinom{m}{3}$.
\end{lemma}

\begin{proof}
Every triple of distinct vertices in $\mathcal{P}$ is mutually adjacent, hence forms a triangle.
\end{proof}

\smallskip

\textbf{Case A. $\dim L'=0$ - The abelian algebra.} Recall from \cite[Proposition 4.2]{TZG} that the set of vertices of $\Gamma(L)$ is $\mathcal{L} \sqcup \mathcal{P}$; two lines are never adjacent, and two distinct planes are always adjacent; a line $A$ and a plane $B$ are adjacent if and only if $A \nsubseteq B$, and finally, $|\mathcal L| = |\mathcal P| = q^2+q+1$. We now count the triangles in $\Gamma(L)$ according to the types of their
vertices.

\begin{proposition}\label{Abelian}
The number of triangles in $\Gamma(L)$ is
\[\mathfrak{t}(\Gamma(L)) = \frac{q(q+1)(4q^4+2q^3+q^2-3q-1)}{6}.\] 
\end{proposition}

\begin{proof}
We classify triangles according to the number of lines and planes appearing
among their vertices.

\smallskip

\noindent
\textbf{1. Three planes.} By Lemma \ref{T3P} with $m=q^2+q+1$ we have that the number of such triangles is
\begin{equation}\label{AbelianT}
T_{3P} = \binom{q^2+q+1}{3}.
\end{equation}

\noindent \textbf{2. Two planes and one line.} Two distinct planes $P_1, P_2$ intersect in a unique line, so the number of lines contained in $P_1\cup P_2$ is $(q+1)+(q+1)-1 = 2q+1$. Then, the number of lines outside both is $q^2+q+1 -(2q+1) = q^2-q$. Since every such line is adjacent to both planes:
\begin{equation}\label{AbelianT2}
T_{2PL} = \binom{q^2+q+1}{2}(q^2-q).
\end{equation}  

\noindent \textbf{3. At least two lines.} Since no two lines are adjacent,  $T_{P2L} = T_{3L} =0$.

Combining the above counts \eqref{AbelianT} and \eqref{AbelianT2} yields
\[\mathfrak{t}(\Gamma(L)) = \frac{(q^2+q+1)(q^2+q)(4q^2-2q-1)}{6}.\]
\end{proof}

\textbf{Case B1. $\dim L'=1$ - The nonabelian nilpotent algebra (The Heisenberg algebra).}
Let $B= \{e,f,h\}$ be a basis for $L$, with brackets $[e,f]=h$, and $[e,h]=[f,h]=0$. Let $Z(L) = L' = \langle h\rangle = \F_q h$ be its center and derived algebra. Recall from \cite[Proposition 4.3]{TZG} that: the set of vertices of $\Gamma(L)$ are the non-central lines and the two-dimensional Lie subalgebras; the central line $\F_q z$ is an isolated point; two distinct planes are always adjacent; two non-central lines are adjacent if and only if they span a two-dimensional subspace not containing $h$, and finally, a non-central line is adjacent to a plane if and only if it is not contained in that plane.

We now determine the number of triangles in $\Gamma(L)$.

\begin{proposition}\label{Heisenberg}
Let $L$ be the Heisenberg Lie algebra over $\F_q$. Then
\[t(\Gamma(L)) = \frac{q(q-1)(q+1)^4}{6}.\]
\end{proposition}

\begin{proof}
We count triangles according to the types of their vertices.

By \cite[Proposition 4.3]{TZG}: the central line $\langle h\rangle$ is an isolated vertex; $|P|=q+1$ and the induced subgraph on $P$ is $K_{q+1}$; a non-central line is adjacent to a plane if and only if it is not contained in it; two non-central lines are adjacent if and only if they span a subspace not containing $h$.

\smallskip

\noindent
\textbf{1. Three planes.} By Lemma \ref{T3P} with $m=q+1$:
\[T_{3P} = \binom{q+1}{3}.\]

\smallskip

\noindent
\textbf{2. Two planes and one non-central line.} Every pair $P_1, P_2\in \mathcal{P}$ satisfies $P_1\cap P_2 = \F_q h = \langle h\rangle$, so  the number of non-central lines contained in $P_1\cup P_2$ equals $q+q=2q$. Then the $q^2-q$
non-central lines outside both are adjacent to both planes, giving
\[T_{2PL} = \binom{q+1}{2}(q^2-q).\]

\noindent
\textbf{3. One plane and two non-central lines.} The $q^2+q$ non-central lines partition into $q+1$ classes of size $q$ (each class corresponding to one plane).  Fix $P\in \mathcal{P}$; the $q^2$ non-central lines outside $P$ form $q$ classes of size $q$.  Two lines are adjacent if and only if they lie in different classes, so the induced graph is $K_{q,\ldots,q}$ (with $q$ parts).  The number of edges is $\tbinom{q}{2}q^2$, hence
\[T_{P2L} = (q+1)\binom{q}{2}q^2.\]

\noindent \textbf{4. Three non-central lines.}
Three non-central lines form a triangle if and only if they belong to three distinct parts of the complete multipartite graph induced by the non-central lines. Thus, the number of such triangles is
\[T_{3L} = \binom{q+1}{3}q^3.\]
Adding all contributions gives the stated formula.
\end{proof}

\textbf{Case B2: $\dim L' = 1$, solvable non-nilpotent.} Let $\{x,y,z\}$ be a basis with $[x,y]=y$ and $[x,z]=[y,z]=0$. Set $V = \langle y, z \rangle$, $L' = \langle y \rangle$, $Z(L) = \langle z \rangle$. 

We use throughout the structure of $\Gamma(L)$ established in \cite[Proposition 4.5]{TZG}.
The $q^2+q+1$ lines of $L$ fall into three classes:
\begin{enumerate}
  \item The lines $\langle y\rangle$, and $\langle z\rangle$ adjacent to no other line.
  \item The $q-1$ intermediate lines $I_{\alpha}=\langle\alpha y+z\rangle$,
        $\alpha\in\F_q^* $, each adjacent to every line outside $V$
        but to no line inside $V$.
  \item The $q^2$ lines $X_{a,b}=\langle x+ay+bz\rangle$, $a,b\in\F_q$,
        where $X_{a,b}\sim X_{c,d}$ if and only if $a\neq c$ and $b\neq d$. From now on, we will refer to this type of line as an $X$-line. 
\end{enumerate}
The $2q+1$ planes split into three families:
\begin{enumerate}
\item $V=\langle y,z\rangle$, containing $\langle y\rangle$, $\langle z\rangle$,        and all $q-1$ intermediate lines.

\item $q$ planes of the form $Q_b=\langle y,\,x+bz\rangle$, where $b\in\F_q$. Note that the intersection of any two of such planes is $\langle y\rangle$. Further, every plane $Q_b$ contain the $q$ lines $\{X_{a,b}\mid a\in\F_q\}$.

\item $q$ planes of the form $P_{\alpha}=\langle x+\alpha y,z\rangle$, with 
        $\alpha\in\F_q$, each containing $\langle z\rangle$ and the $q$ lines
        $\{X_{\alpha,t}\mid t\in\F_q\}$.
\end{enumerate}
Note that the planes containing $\langle y\rangle$ are exactly $V$ together with
the $q$ planes $Q_b$, for a total of $q+1$.
A line $A$ is adjacent to a plane $P$ if and only if $A\not\subseteq P$.

\begin{proposition}\label{B2}
Let $L$ be the three-dimensional solvable non-nilpotent Lie algebra over $\F_q$ with $\dim L'=1$. Then
\[\mathfrak{t}(\Gamma(L)) = \frac{q(q+1)(q^4+2q^3+5q^2-6q-2)}{6}.\]

\end{proposition}

\begin{proof}
We classify triangles by vertex type.

\medskip
\noindent\textbf{1. Three planes.} By Lemma \ref{T3P} with $m=2q+1$ we have: $T_{3P}= \binom{2q+1}{3} =\frac{q(4q^2-1)}{3}$.

\smallskip
\noindent\textbf{2. Two planes and one line.}
For a line $A$ and a plane $P$ of $L$,  the structural description is as follows:

\begin{enumerate}
\item[(1)] The lines $\langle y \rangle$ and $\langle z \rangle$ are adjacent
to no other line, but $\langle y \rangle \sim P_\alpha$ for every
$\alpha \in \mathbb F_q$, and $\langle z \rangle \sim Q_b$ for every
$b \in \mathbb F_q$; neither is adjacent to $V$, and $\langle y \rangle$ is
not adjacent to any $Q_b$, nor $\langle z \rangle$ to any $P_\alpha$.
\item[(2)] The $q-1$ intermediate lines $I_\alpha = \langle \alpha y + z
\rangle$, $\alpha \in \mathbb F_q^*$, each adjacent to every line outside
$V$ but to no line inside $V$.
\item[(3)] The $q^2$ lines $X_{a,b} = \langle x + ay + bz \rangle$,
$a,b \in \mathbb F_q$, where $X_{a,b} \sim X_{c,d}$ if and only if $a \neq c$
and $b \neq d$.
\end{enumerate}

We separate into five pair types:
\begin{enumerate}
\item[(i)] Pair $(Q_{b_1},\,Q_{b_2})$. The $q-1$ intermediates lie outside both planes; each $X_{a,b}$ satisfies $X_{a,b}\sim Q_{b_i}$ if and only if $b\neq b_i$, so
$b\notin\{b_1,b_2\}$ gives $q(q-2)$ valid $X$-lines.  In addition, $\langle z \rangle$ is adjacent to
every $Q_b$, hence to both $Q_{b_1}$ and $Q_{b_2}$. Including this line, the count per pair is $q^2-q$. With
$\binom{q}{2}$ pairs: $T_{2PL}^{(i)} = \binom{q}{2}\,q(q-1)$.


\smallskip

\item[(ii)] Pair $(P_{\alpha_1},\,P_{\alpha_2})$. By the symmetric argument
(swap $a$- and $b$-coordinates), together with the contribution of
$\langle y \rangle$ (adjacent to every $P_\alpha$), the count per pair is
also $q^2-q$, giving $T_{2PL}^{(ii)} = T_{2PL}^{(i)}$.

\smallskip

\item[(iii)] Pair $(Q_b,\,P_\alpha)$. In this case, their intersection is the line $X_{\alpha, b}$, and all $q-1$ intermediate lines are adjacent to both.
An $X$-line $X_{a,b'}$ is simultaneously adjacent to $Q_b$, and to $P_\alpha$
if and only if $b'\neq b$ and $a\neq\alpha$, giving $(q-1)^2$ such lines, and
neither $\langle y \rangle$ nor $\langle z \rangle$ is adjacent to both $Q_b$
and $P_\alpha$, since $\langle y \rangle \not\sim Q_b$ and $\langle z \rangle
\not\sim P_\alpha$). Hence $T_{2PL}^{(iii)} = q^3(q-1)$.  
\smallskip

\item[(iv)] Pair $(V, Q_b)$. Only $X$-lines are adjacent to $V$; $\langle z
\rangle$ is not adjacent to $V$, giving $q(q-1)$ lines adjacent to both. Thus, $T_{2PL}^{(iv)}= q^2(q-1)$.
\smallskip

\item[(v)] Pair $(V, P_\alpha)$. Symmetric to (iv) we have  $T_{2PL}^{(v)}= q^2(q-1)$.
\end{enumerate}
Summing the five terms, we have
\[T_{2PL} = q^2(q-1)^2 + q^3(q-1) + 2q^2(q-1) = q^2(q-1) ((q-1)+q+2) = q^2(q-1)(2q+1).\]

\noindent\textbf{3. One plane and two lines.} We compute the number of edges among lines adjacent to a plane $P$, for each plane type. We denote it by $e_P$. Then we have $T_{P2L}=\sum_{P\in \mathcal{P}} e_P$. Fix a plane $P\in \mathcal{P}$.
\begin{enumerate}
\item[(i)] $P=V$. Lines adjacent to $V$ are the $q^2$ $X$-lines.  Non-adjacent pairs share an $a$-value ($q\binom{q}{2}$ pairs) or a $b$-value (another $q\binom{q}{2}$ pairs), with the two events disjoint.  Hence
\begin{equation}\label{eV-B2}
  e_{V}=\binom{q^2}{2}-2q\binom{q}{2}=\frac{q^2(q-1)^2}{2}.
\end{equation}

\smallskip

\item[(ii)] $P= Q_b$. Lines adjacent to $Q_{b}$: $q-1$ intermediates and $q(q-1)$ $X$-lines with $b'\neq b$.  No two intermediates are adjacent ($0$ edges); each intermediate is adjacent to every $X$-line ($q(q-1)^2$ edges); among $X$-lines, non-adjacent pairs share an $a$-value ($q\binom{q-1}{2}$) or a $b'$-value ($(q-1)\binom{q}{2}$), giving
\begin{equation}\label{e2X-B2}
e_{2X} =\binom{q(q-1)}{2}-q\binom{q-1}{2}-(q-1)\binom{q}{2} =\frac{q(q-1)^2(q-2)}{2}.
\end{equation}
Hence $e_{Q_{b}}=q(q-1)^2+\frac{q(q-1)^2(q-2)}{2}=\frac{q^2(q-1)^2}{2}$.

\smallskip

\item[(iii)] $P= P_{\alpha}$. By the symmetric argument (swapping $a$- and $b$-coordinates) follows $e_{P_{\alpha}} =\frac{q^2(q-1)^2}{2}$.
\end{enumerate}
Every plane gives $e_{P}=\frac{q^2(q-1)^2}{2}$, so
\[T_{P2L}=(2q+1) \frac{q^2(q-1)^2}{2} =\frac{q^2(q-1)^2(2q+1)}{2}.\]
Note that neither $\langle y\rangle$ nor $\langle z \rangle$ is adjacent to any line, their presence
among the vertices adjacent to a fixed plane $P$ contributes no new edges. 

\smallskip
\noindent\textbf{4. Three lines.} Since $\langle y \rangle$ and $\langle z
\rangle$ are adjacent to no line, $\langle y\rangle$, $\langle z\rangle$ are in no triangle; no two intermediates are adjacent. Then, in this case, only two options arise.
\begin{enumerate}
\item[(i)] One intermediate line and two $X$-lines. Any adjacent $X$-pair forms a triangle with each of the $q-1$ intermediates, giving
\[T_{3L}^{(i)} = (q-1) \frac{q^2(q-1)^2}{2}=\frac{q^2(q-1)^3}{2}.\]

\item[(ii)] Three $X$-lines. $X_{a,b},X_{c,d},X_{e,f}$ are mutually adjacent if and only if $a,c,e$ are pairwise distinct and $b,d,f$ are pairwise distinct.  Counting ordered triples and dividing by $3!$: 
\[T_{3L}^{(ii)} =\frac{(q(q-1)(q-2))^2}{6}=\frac{q^2(q-1)^2(q-2)^2}{6}.\]
\end{enumerate}
Combining the two sub-cases, we have
\[T_{3L}=\frac{q^2(q-1)^3}{2}+\frac{q^2(q-1)^2(q-2)^2}{6}.\]
Adding all contributions gives the stated formula.
\end{proof}

\subsection{Comaximal graph of a 3-dimensional Lie algebra}
A unifying feature of the cases considered below is that, whenever $\dim L' \geq 1$, the algebra $L$ admits a decomposition $L = \langle x \rangle \ltimes V$, where $V$ is a two-dimensional ideal. The structure of $L$ is then determined by the Endomorphism $T = \ad\, x\mid_V$. The different cases correspond to different types of the operator $T$. For example, in the case $\dim L' = 2$, the operator $T$ is invertible, and its Jordan type over $\F_q$ determines the number of two-dimensional subalgebras. This notation and decomposition will be used without further comment throughout all cases below.

\smallskip

\textbf{Case C. $\dim L'=2$ - Solvable non-nilpotent algebras.}
Here, $L'$ is a 2-dimensional abelian ideal of $L$. In the notation of the introduction of this subsection, the structure of $\Gamma(L)$ depends on the Jordan type of the operator $T = \ad x |_V$ over $\F_q$, which determines the number and type of two-dimensional subalgebras of $L$. We distinguish four subcases according to the canonical form of $T$ over $\F_q$.

\begin{lemma}\cite[Lemma 4.6]{TZG} \label{2dim-subalgebras}
Let $L=Fx+V$ where $V=Fv_1+Fv_2$ and $[x,v_1]=\alpha v_1+\beta v_2$, $[x,v_2]=\gamma v_1+\delta v_2$ where $A=\left(\begin{matrix} \alpha&\beta\\ \gamma & \delta \end{matrix}\right)$ is a non-singular matrix. Then
\begin{itemize}
\item[(i)] If $\ad x$ has no eigenvector in $V$, the only two-dimensional subalgebra of $L$ is $V$; 
\item[(ii)] if $\ad x$ has an eigenvector in $V$, the two-dimensional subalgebras of $L$ are $V$ or of the form $Fx+Fv$ where $v$ is an eigenvector, and $Fv+F(x+v_3)$ where $v$ is an eigenvector and $v_3$ is linearly independent of $v$.
\end{itemize}
\end{lemma}

\smallskip

So there are four cases for algebras in this case.
\smallskip

\noindent {\emph Case 1: The characteristic equation of $A$ is irreducible.} Then the only two-dimensional subalgebra is $V$.  
\smallskip

\noindent {\emph Case 2: There are two different eigenvalues $\lambda$ and $\mu$.} Then, replacing $x$ by $\lambda^{-1}x$, we have a basis such that $[x,v_1]=v_1$, $[x,v_2]=\mu v_2$. The two-dimensional subalgebras are $V$, $Fv_1+F(x+\alpha v_2)$ and $Fv_2+F(x+\alpha v_1)$.  

\smallskip

\noindent {\emph Case 3: There is only one eigenvalue $\lambda$ and $A$ is not diagonable.} Then there is a basis such that $[x,v_1]=\lambda v_1$, $[x,v_2]=v_1+\lambda v_2$. The two-dimensional subalgebras are $V$, and $Fv_1+F(x+\alpha v_2)$.  
\smallskip

\noindent {\emph Case 4: $A$ is a scalar matrix, $\lambda I_2$.}  Then, replacing $x$ by $\lambda^{-1} x$, there is a basis such that $[x,v_1]=v_1$, $[x,v_2]=v_2$.

\begin{proposition}\cite[Proposition 4.7]{TZG}\label{dim3-derived2}
Let $L$ be a three-dimensional Lie algebra over the finite field $\F_q$ such that $\dim L' = 2$. As before, let $\mathcal{L}$ denote the set of one-dimensional Lie subalgebras of $L$, and let $\mathcal{P}$ denote the set of two-dimensional Lie subalgebras of $L$. Then 
\begin{enumerate}
\item $\mathcal{L}$ has $q^2 + q + 1$ elements.
\item The number of two-dimensional Lie subalgebras is determined by the canonical form of $\ad x$ over $\F_q$ as follows:
\[|\mathcal{P}| = \begin{cases}
    1 &  \text{in case 1}\\
    1+2q &  \text{in case 2}\\
    1+q &  \text{in case 3}\\
    1+q+q^2 &  \text{in case 4}.
\end{cases}\]

\item The following rules hold in all four cases:
\begin{enumerate}
    \item If $A, B\in \mathcal{P}$ with $A\neq B$, then $A\sim B$. Hence, the induced subgraph on $\mathcal{P}$ is a complete graph.
    \item If $A \in \mathcal{L}$ and $P \in \mathcal{P}$, then $A \sim P$ if and only if $A \not\subseteq P$.
    \item If $A, B \in \mathcal{L}$ are distinct, then $A \not\sim B$ if and only if $A$ and $B$ are contained in a common plane of $\mathcal{P}$.
\end{enumerate}
    In particular:
\begin{enumerate}
    \item[(i)] In Case 1, $\Gamma(L)[\mathcal{L}]$ is the join of an independent 
    set on the $q+1$ lines of $V$ with a complete graph $K_{q^2}$ on the 
    lines outside $V$.
    \item[(ii)] In Case 3, the unique eigenline $\langle v_1\rangle$ is an isolated vertex of $\Gamma(L)$, since it is contained in every plane of $\mathcal{P}$.
    \item[(iii)] In Case 4, no two lines are adjacent, since every pair of lines 
    spans a plane in $\mathcal{P}$.
\end{enumerate}
\end{enumerate}
\end{proposition}

\smallskip

\begin{proposition}\label{irreducible}
Let $L$ be a three-dimensional Lie algebra over $\F_q$ with $\dim L'=2$ such that
the characteristic polynomial of $\ad x |_V$ is irreducible over $\F_q$
for some $x\notin L'$.  Then 
\[\mathfrak{t}(\Gamma(L)) = \frac{q^2(q^2-1)(q^2+3q+4)}{6}.\]
\end{proposition}

\begin{proof}
By \cite[Proposition 4.7]{TZG}, there is a unique plane $V=L'$; every line outside $V$ is adjacent to every other vertex; lines inside $V$ are adjacent only to lines outside $V$; no two lines inside $V$ are adjacent. In particular, there are no triangles involving two or three planes.

\medskip
\noindent\textbf{1. One plane and two lines.}
A line is adjacent to $V$ if and only if it is not contained in $V$. The $q^2$ lines outside $V$ each form a complete subgraph, so $T_{V2L} = \binom{q^2}{2} = \frac{q^2(q^2-1)}{2}$.

\smallskip
\noindent\textbf{2. Three lines.}
Two lines are non-adjacent if and only if both lie inside $V$, so three lines are
mutually adjacent if and only if at most one of them lies in $V$. 
\begin{enumerate}
\item[(i)] If all three lines are outside $V$ then $T_{3L}^{(i)} = \binom{q^2}{3} = \frac{q^2(q^2-1)(q^2-2)}{6}$.
    
\item[(ii)] If two lines are outside $V$, and one is inside $V$, then each of the $q+1$ lines inside $V$ is adjacent to all $q^2$ lines outside $V$, giving $T_{3L}^{(ii)} = (q+1)\binom{q^2}{2} = \frac{(q+1)q^2(q^2-1)}{2}$.
\end{enumerate}
Therefore, the number of triangles formed by three lines is
\[T_{3L} = \frac{q^2(q^2-1)(q^2-2)}{6} + \frac{(q+1)q^2(q^2-1)}{2}.\]
Finally, adding all contributions gives the stated formula.
\end{proof}

\begin{proposition}\label{2EigenVal}
Let $L$ be a three-dimensional Lie algebra over $\F_q$ with $\dim L'=2$ such that
$\ad x |_V$ has two distinct eigenvalues in $\F_q$ for some $x\notin L'$.
Then 
\[\mathfrak{t}(\Gamma(L)) = \frac{q(q+1)(q^4+2q^3+5q^2-6q-2)}{6}.\]
\end{proposition}

\begin{proof}
Let $L_B$ denote the algebra of Proposition \ref{B2} with basis $\{x, y, z\}$, and $[x,y]=y$, $[x,z]=[y,z]=0$. 
If $L$ has basis $\{x, v_1, v_2\}$ with $[x,v_1]=v_1$ and $[x,v_2]=\mu v_2$, then the mapping $x\mapsto x$, $v_1\mapsto y$ and $v_2\mapsto z$, and extending it linearly, gives a lattice isomorphism between the two subalgebra lattices. Then both graphs have the same number of triangles. 

 \end{proof}

\begin{proposition}\label{1EigenVal}
Let $L$ be a three-dimensional Lie algebra over $\F_q$ with $\dim L'=2$ such that
$\ad x |_V$ has a unique eigenvalue and is non-diagonalisable over $\F_q$
for some $x\notin L'$.  Then  
\[ \mathfrak{t}(\Gamma(L)) = \frac{q(q-1)(q+1)^4}{6}.\]
\end{proposition}

\begin{proof}
In this case, $L$ has the same subalgebra lattice as the three-dimensional Heisenberg algebra

\smallskip
\end{proof}

\begin{proposition}\label{CaseC4}
Let $L$ be a three-dimensional Lie algebra over $\F_q$ with $\dim L'=2$ such that
$\ad x |_V$ is a nonzero scalar matrix for some $x\notin L'$. Then 
\[\mathfrak{t}(\Gamma(L)) = \frac{q(q+1)(4q^4+2q^3+q^2-3q-1)}{6}.\] 
\end{proposition}

\begin{proof}
Since the Lie algebra in Proposition \ref{Abelian} is abelian and, in this case, $L$ is almost abelian, so they have the same subalgebra lattice. 

\end{proof}

\textbf{Case D. $\dim L'=3$ - the perfect case.}

\begin{theorem} \cite[Theorem 4.9]{TZG}
Let $L$ be a perfect three-dimensional Lie algebra over a finite field $\F_q$ where $q\neq 2$. Then $L\cong \mathfrak{sl}_2(\F_q)$ (the split case) or $L\cong \mathfrak{su}_2(\F_q)$. In the latter case, $L$ has a basis $e_1,e_2,e_3$ with $[e_1,e_2]=e_3$, $[e_2,e_3]=e_1$ and $[e_3,e_1]=e_2$ (the non-split case).
\end{theorem}

From now on, we assume $q$ is odd. Since every proper subalgebra of $\mathfrak{su}_2(\F_q)$ is one-dimensional, the graph is complete. Then the interesting case is the split one. Throughout, we assume that $L=\mathfrak{sl}_2(\F_q)$.  We freely use the structural results of \cite{TZG}: the partition
\[V(\Gamma(L)) = \mathcal{B}\sqcup \mathcal{N}\sqcup L_s \sqcup \mathcal{L}_{ns}\]
of \cite[Proposition 4.10]{TZG}, where $\mathcal{B}$ is the set of $q+1$ Borel subalgebras, $\mathcal{N}$ is the set of $q+1$ nilpotent lines, $\mathcal{L}_s$ is the set of $\tfrac{q(q+1)}{2}$ split semisimple lines, and $\mathcal{L}_{ns}$ is the set of $\frac{q(q-1)}{2}$ non-split semisimple lines; the adjacency rules of \cite[Proposition 4.11]{TZG} (two Borels are always adjacent; a line is adjacent to a Borel if and only if it is not contained in it; two lines are adjacent if and only if no Borel contains both); the containment facts of \cite[Lemma 4.14]{TZG} (each Borel contains exactly one nilpotent line and $q$ split lines; each nilpotent line lies in exactly one Borel; each split line lies in exactly two Borels; non-split lines lie in no Borel).  

\begin{lemma}\label{bijection}
There is a bijection between $N\sqcup \mathcal{L}_s$ and the set of subsets
of $\mathcal{B}$ of size 1 or 2: a nilpotent line corresponds to the unique Borel containing it, and a split line corresponds to the (unordered) pair of Borels containing it.  Under this bijection:
\begin{enumerate}
  \item two lines in $\mathcal{N}\sqcup \mathcal{L}_s$ are adjacent in $\Gamma(L)$ if and only if their corresponding subsets of $\mathcal{B}$ are disjoint;
  \item a line in $\mathcal{N}\sqcup \mathcal{L}_s$ is adjacent to a Borel $B$ if and only if its corresponding subset does not contain $B$.
\end{enumerate}
In particular, two distinct Borels share exactly one split line, and
$|N\sqcup L_s| =(q+1) +  \tfrac{q(q+1)}{2} =\binom{q+2}{2}$.
\end{lemma}

\begin{proof}
It follows immediately from \cite[Lemma 4.14]{TZG} and the adjacency rules of \cite[Proposition 4.11]{TZG}. Further, since each split line corresponds to a distinct pair of Borels and $|\mathcal{L}_s|= \tfrac{q(q+1)}{2}$, every pair of Borels arises from exactly one split line, proving the claim.
\end{proof}

Non-split lines lie in no Borel, hence are adjacent to every Borel and (since no Borel can contain a pair including a non-split line) to every other line as well; in particular, $L_{ns}\sqcup\mathcal{B}$ is a clique, recovering \cite[Corollary 4.16 (3)]{TZG}.

\begin{example}
We illustrate for $q=3$ the bijection of Lemma \ref{bijection} using the information in \cite[Example 4.13]{TZG} (Tables 1 and 2 there), combined into
a single table.  Recall that the four Borel subalgebras for $q=3$ are $B=\F_3 x + \F_3 h$, $B(0)=\F_3 h+\F_3 y$, $B(1)=\F_3(h+x)+\F_3(y+x)$, and $B(2) =\F_3(h+2x) +\F_3(y+x)$.
\smallskip

\begin{table}[H]
\centering
\caption{Lines of $\mathfrak{sl}_2(\F_3)$, their type, and (for $\mathcal{N}\sqcup \mathcal{L}_s$) the subset of $\mathcal{B}$ assigned by the bijection of Lemma \ref{bijection}.}
\smallskip
\begin{tabular}{l|c|c|c|c}
\hline
Line & Generator & $\Delta$ & Type & Subset of $\mathcal{B}$ \\
\hline
$L_1$    & $h$        & $1$ & Split     & $\{B,\,B(0)\}$ \\
$L_2$    & $x$        & $0$ & Nilpotent & $\{B\}$ \\
$L_3$    & $y$        & $0$ & Nilpotent & $\{B(0)\}$ \\
$L_4$    & $h+x$      & $1$ & Split     & $\{B,\,B(1)\}$ \\
$L_5$    & $h+y$      & $1$ & Split     & $\{B(0),\,B(2)\}$ \\
$L_6$    & $x+y$      & $1$ & Split     & $\{B(1),\,B(2)\}$ \\
$L_7$    & $h+2x$     & $1$ & Split     & $\{B,\,B(2)\}$ \\
$L_8$    & $h+2y$     & $1$ & Split     & $\{B(0),\,B(1)\}$ \\
$L_9$    & $x+2y$     & $2$ & Nonsplit  & -- \\
$L_{10}$ & $x+y+h$    & $2$ & Nonsplit  & -- \\
$L_{11}$ & $2x+y+h$   & $0$ & Nilpotent & $\{B(1)\}$ \\
$L_{12}$ & $x+2y+h$   & $0$ & Nilpotent & $\{B(2)\}$ \\
$L_{13}$ & $x+y+2h$   & $2$ & Nonsplit  & -- \\
\hline
\end{tabular}
\end{table}

The last column is read off directly from \cite[Tables 1 and 2]{TZG}: each nilpotent or split line is assigned the set of Borels of Table 2 that contain it (nonsplit lines lie in no Borel, so the bijection does not apply to them, marked --).  As stated in Lemma \ref{bijection}, the four nilpotent lines correspond exactly to the four singletons of $\mathcal{B}$, and the six split lines correspond exactly to the six pairs.
\end{example}

\begin{proposition}\label{CaseD}
Let $L=\mathfrak{sl}_2(\F_q)$ with $q$ odd.  Then
\[\mathfrak{t}(\Gamma(L)) = \frac{q^2(q^2-1)(q^2+3q+3)}{6}.\]
\end{proposition}

\begin{proof}
We count triangles according to the number of Borel subalgebras
among their three vertices.

\smallskip
\noindent\textbf{1. Three Borels.} Every pair of Borels is adjacent, so
\[  T_{3\mathcal{B}} =\binom{q+1}{3}=\frac{q(q^2-1)}{6}.\]

\noindent\textbf{2. Two Borels and one line.} Fix two distinct Borels $B_1,B_2$, and let $A$ be any line.  
\begin{enumerate}
\item[(i)] If $A\in \mathcal{L}_{ns}$, it is automatically adjacent to both. 

\item[(ii)] If $A\in \mathcal{N}\sqcup \mathcal{L}_s$, by Lemma \ref{bijection} (2), it is adjacent to both if and only if its corresponding subset of $\mathcal{B}$ avoids $\{B_1, B_2\}$, that is, it is a subset of size 1 or 2 of the remaining $q-1$ Borels: $(q-1)+ \tbinom{q-1}{2} =\tbinom{q}{2}$ choices.
Summing over both types and then over all $\tbinom{q+1}{2}$ pairs of
Borels:
\[T_{2\mathcal{B}L} = \binom{q+1}{2} \left(\binom{q}{2}+ \frac{q(q-1)}{2}\right) =\frac{q^2(q^2-1)}{2}.\]
\end{enumerate}

\noindent\textbf{3. One Borel and two lines.}
Fix $B\in\mathcal{B}$.  We compute $e_{B}$, the number of adjacent pairs of lines that are both adjacent to $B$, splitting by type.
\begin{enumerate}
\item[(i)] Both lines in $\mathcal{N}\sqcup \mathcal{L}_s$. By Lemma \ref{bijection}, such lines correspond to subsets of size 1 or 2 of the $q$-set $\mathcal{B}\setminus\{B\}$, and two are mutually adjacent (and automatically adjacent to $B$) if and only if their subsets are disjoint. Then 
\[e_B^{(i)}=\underbrace{\binom{q}{2}}_{(\mathcal{N},\mathcal{N})} + \underbrace{q\binom{q-1}{2}}_{(\mathcal{N},\mathcal{L}_s)} +\underbrace{\frac{1}{2}\binom{q}{2}\binom{q-2}{2}}_{(\mathcal{L}_s, \mathcal{L}_s)}  =\frac{q(q-1)(q^2-q+2)}{8}.\]

\item[(ii)] One line in $\mathcal{N}\sqcup \mathcal{L}_s$, and one in $\mathcal{L}_{ns}$. The line in $\mathcal{N}\sqcup \mathcal{L}_s$ adjacent to $B$ corresponds to a subset of size 1 or 2 of $\mathcal{B}\setminus\{B\}$. Then, there are $q + \binom{q}{2} = \binom{q+1}{2}$ such lines. Every non-split line is adjacent to it automatically, so
\[e_B^{(ii)} = \binom{q+1}{2} \frac{q(q-1)}{2} =\frac{q^2(q^2-1)}{4}.\]
\item[(iii)] Both lines in $\mathcal{L}_{ns}$. Always adjacent, therefore, $e_B^{(iii)} =\binom{\frac{q(q-1)}{2}}{2} = \frac{q(q-2)(q^2-1)}{8}$.
\end{enumerate}

Summing the three contributions we have $e_B = \frac{q^{3}(q-1)}{2}$. Over all $q+1$ Borels we have the total number:
\[T_{\mathcal{B}2L}=(q+1) \frac{q^{3}(q-1)}{2} = \frac{q^{3}(q^2-1)}{2}.\]

\noindent\textbf{4. Three lines.}
We decompose by the number of non-split lines among the three.
\begin{enumerate}
\item[(i)] Three non-split. $T_{3ns}=\binom{\frac{q(q-1)}{2}}{3}$.

\item[(ii)] Two non-split, and one in $\mathcal{N}\sqcup \mathcal{L}_s$. By Lemma \ref{bijection}, the third vertex can be any of the $\binom{q+2}{2}$ lines of $\mathcal{N}\sqcup \mathcal{L}_s$. Then 
\[T_{2nsL} = \binom{\tfrac{q(q-1)}{2}}{2} \frac{(q+1)(q+2)}{2}.\]

\item[(iii)] One non-split, two in $\mathcal{N}\sqcup \mathcal{L}_s$.
The non-split line is automatically adjacent to both; the other two
must be mutually adjacent. If we define
\[\alpha(q) := |\{\{A_1,A_2\}\subseteq \mathcal{N}\sqcup \mathcal{L}_s \mid A_1\sim A_2\}|,\]
then by Lemma \ref{bijection} counts disjoint pairs of subsets of size 1 or 2 from $\mathcal{B}$. Thus, 
\begin{align*}
\alpha(q) & = \underbrace{\binom{q+1}{2}}_{(\mathcal{N},\mathcal{N})}   + \underbrace{(q+1)\binom{q}{2}}_{(\mathcal{N},\mathcal{L}_s)} + \underbrace{\tfrac{1}{2} \binom{q+1}{2} \binom{q-1}{2}}_{(\mathcal{L}_s, \mathcal{L}_s)}\\ 
& =\frac{q(q+1)(q^2+q+2)}{8}.
\end{align*}
Hence,
\[T_{ns2L} = \frac{q(q-1)}{2} \alpha(q) = \frac{q^2(q^2-1)(q^2+q+2)}{8}.\]

\item[(iii)] No non-split lines: three lines in $N\sqcup L_s$.
Define
\[\beta(q) := |\{\{A_1,A_2,A_{3}\}\subseteq \mathcal{N}\sqcup \mathcal{L}_s\mid A_1 \sim A_2 \sim A_3 \sim A_1\}|,\]
the number of mutually adjacent triples within $\mathcal{N}\sqcup \mathcal{L}_s$.  By Lemma \ref{bijection}, this is the number of ways to choose three pairwise-disjoint subsets of size 1 or 2 from $\mathcal{B}$.  Decomposing by the sizes of the three subsets:
\begin{align*}
\beta(q) & = \binom{q+1}{3} + \binom{q+1}{2}\binom{q-1}{2} + \frac{1}{2} (q+1) \binom{q}{2} \binom{q-2}{2} \\
& \hskip0.5cm + \frac{1}{6} \binom{q+1}{2} \binom{q-1}{2} \binom{q-3}{2}\\
& = \frac{q(q^2-1)}{6} + \frac{q(q-2)(q^2-1)}{4} + \frac{q(q-3)(q-2)(q^2-1)}{8}\\ & \hskip0.5cm +\frac{q(q-4)(q-3)(q-2)(q^2-1)}{48}.
\end{align*}
Combining the four sub-cases:
\[T_{3L} = \frac{q(q-1)^2(q+1)(q^2+q+1)}{6}.\]
\end{enumerate}
Adding all contributions gives the stated formula.
\end{proof}

\section{The comaximal graphs for certain 4-dimensional Lie algebras over finite fields}

\begin{lemma}\label{g-adjacency}
Let $L$ be a four-dimensional Lie algebra over a field $\mathbb{F}$,  and let $\mathcal{L}$ denote the set of lines (one-dimensional Lie subalgebras), $\mathcal{P_2}$ the set of 2-planes (two-dimensional Lie subalgebras) of $L$ and $\mathcal{P_3}$ the set of hyperplanes (three-dimensional Lie subalgebras) of $L$.
\begin{enumerate}[(1)]
\item If $A, B \in \mathcal{P}_3$ with $A \neq B$, then $A \sim B$ in $\Gamma(L)$. In particular, the induced subgraph 
$\Gamma(L)[\mathcal{P}_3]$ is always a complete graph.
\item If $A, B \in \mathcal{P}_2$ with $A\cap B=0$, then $A \sim B$ in $\Gamma(L)$.
\item If $A \in \mathcal{L}$ and $B \in \mathcal{P}_3$, then $A \sim B$ in $\Gamma(L)$ if and only if $A \cap B=0$.
\item If $A, B\in \mathcal{P}_2$, the set of two-dimensional subalgebras, then $A\sim B$ if and only if they don't belong to the same maximal subalgebra.
\item If $A\in \mathcal{L}$ and $B\in \mathcal{P}_2$ then $A\sim B$ if and only if they don't belong to the same maximal subalgebra.
\end{enumerate}

\begin{proof}
\begin{enumerate}[(1)]
\item Since $A$ and $B$ are distinct three-dimensional subspaces of the four-dimensional space $L$, we have $\dim(A + B) = 4$. Hence $\langle A, B \rangle = A + B = L$, and therefore $A \sim B$.
 \item If $A$ and $B$ are two-dimensional subspaces of the four-dimensional space $L$ with $A\cap B=0$ we have $\dim(A + B) = 4$. Hence $\langle A, B \rangle \supseteq A + B = L$, and therefore $A \sim B$. 
\item Let $A\in\mathcal{L}$ and $B\in\mathcal{P}_3$. If $A\subseteq B$, then $\langle A,B\rangle =B\neq L$, so $A$ and $B$ are not adjacent. Conversely, if $A\not\subseteq B$, then $\langle A,B\rangle = L$, and $A$ and $B$ are adjacent.
\item Let $A,B\in \mathcal{P}_2$. Then $\langle A,B\rangle=L$ precisely when it is not inside a maximal subalgebra.
\item This follows as in (4).
\end{enumerate}
\end{proof}
\end{lemma}

\subsection{Comaximal graph of a 4-dimensional nilpotent Lie algebra}
There are three types if 4-dimensional nilpotent algebra over any field of characteristic different from two.
\smallskip

\textbf{Case 1. $\dim L'=0$ - The abelian algebra.} In this case, every subspace of $L$ is a Lie subalgebra. Hence, the proper nonzero Lie subalgebras of $L$ are all the 1-dimensional, 2-dimensional and 3-dimensional vector subspaces of the vector space $L$. The number of lines and hyperplanes in $L$ are $q^3+q^2+q+1$ and the number of 2-planes is $(q^2+1)(q^2+q+1)$.  

\begin{proposition}\label{ab}
The comaximal graph $\Gamma(L)$ has vertex set $\mathcal{V} = \mathcal{L} \sqcup \mathcal{P}_2 \sqcup \mathcal{P}_3$, and adjacency is described as follows:
\begin{enumerate}
\item If $A, B\in \mathcal{P}_3$ with $A\neq B$, then they are adjacent. Thus, the induced subgraph on $\mathcal{P}_3$ is the complete graph $K_{q^3+q^2+q+1}$.
\item If $A, B\in \mathcal{P}_2$ then they are adjacent if and only if $A\cap B=0$.
\item If $A, B\in \mathcal{L}$, then they are never adjacent. Thus, the induced subgraph on $\mathcal{L}$ is an independent set.
\item If $A\in \mathcal{L}$ and $B\in \mathcal{P}_2$, then they are never adjacent. Thus, the induced subgraph on $\mathcal{L}$ is an independent set.
\item For $A\in \mathcal{L}$ and $B\in \mathcal{P}_3$ one has $A\sim B$ if and only if $A\cap B=0$.
\end{enumerate}
\end{proposition}

\begin{proof}
Clearly the vertex set $\mathcal{V}$ of $\Gamma(L)$ is $\mathcal{V} = \mathcal{L} \sqcup \mathcal{P}_2 \sqcup \mathcal{P}_3$.

\smallskip

\noindent (1) By Lemma \ref{g-adjacency} (1), the induced subgraph on $\mathcal{P}_3$ is complete.


\smallskip

\noindent (2) Let $A, b\in \mathcal{P}_2$. Then they are adjacent if and only if $L=\langle a,B\rangle = A+B$, and this holds if and only if $A\cap B=0$.

\smallskip

\noindent (3) Let $A, B\in \mathcal{L}$ with $A\neq B$. Then $\dim(A+B)=2$. Since $L$ is abelian, $\langle A,B\rangle =A+B$, so $\langle A,B\rangle \neq L$. Hence, $A$ and $B$ are not adjacent. Since this holds for every pair of distinct lines, the induced subgraph on $\mathcal{L}$ is an independent set.

\smallskip

\noindent (4) Let $A\in \mathcal{L}$ and $B\in \mathcal{P}_2$. Then $\langle A,B\rangle =A+B\neq L$ since $\dim (A+B)=3$.

\smallskip

\noindent (5) This follows immediately from Lemma \ref{g-adjacency} (3).


\smallskip

\end{proof}

 \textbf{Case 2. $\dim L'=1$ - The Heisenberg algebra plus a central line.}
 Here we have $L=H_3(\F)\oplus \F e_4$ with basis $e_1,e_2,e_3,e_4$ and $[e_1,e_2]=e_3$. Put $V=\F e_1+\F e_2$, $Z=\F e_3+\F e_4$. Let $U$ be a subalgebra of dimension greater than or equal to two. If $e_3\notin U$ then the projection of $U$ onto $V$ can have dimension at most one, so $U$ is abelian. 
\begin{proposition}\label{heis+}  The proper nonzero subalgebras of $L$ are as given below.
\begin{enumerate}
 \item The three-dimensional subalgebras are of two types, as: $\mathcal{P}_{3,1}(\lambda,\mu)= \F(e_1+\lambda e_4)+\F(e_2+\mu e_4)+\F e_3$ (these are Heisenberg subalgebras) and $\mathcal{P}_{3,2}(\lambda.\mu)=Z+\F(\lambda e_1+\mu e_2)$ (these are abelian subalgebras). If $\F=\F_q$, there are $q^2$ of the former and $q+1$ of the latter, making $q^2+q+1$ in total 
 \item There are no two-dimensional maximal subalgebras. The subalgebras in $\mathcal{P}_{3,1}(\lambda,\mu)$ which are two-dimensional must contain $e_3$ and are therefore of the form $\F e_3+\F(\alpha e_1+\beta e_2+\gamma e_4)$. There are $q^2+q+1$ of these and they all belong to $\mathcal{P}_{3,2}(\lambda,\mu)$ also. As $\mathcal{P}_{3,2}(\lambda,\mu)$ is abelian, every two-dimensional subspace of it is a subalgebra. If $\F=\F_q$, there are $q^3+2q^2+q+1$ subalgebras of this type.
 \item  If $A, B \in \mathcal{P}_3$ with $A \neq B$, then $A \sim B$. If $\F=\F_q$, $\Gamma(L)[\mathcal{P}_3]\cong K_{q^2q+1}$; that is, the induced subgraph on $\mathcal{P}$ is the complete graph $K_{q^2+q+1}$.
\item If $A,B\in \mathcal{P}_2$, the set of two-dimensional subalgebras, then $A\sim B$ if and only if they don't belong to the same plane in $\mathcal{P}_3$.
\item If $A\in \mathcal{L}$ and $B\in \mathcal{P}_2$ then $A\sim B$ if and only if they don't belong to the same plane in $\mathcal{P}_3$.
 \end{enumerate}
\end{proposition} 
\begin{proof} The comaximal graph $\Gamma(L)$ has vertex set $\mathcal{V} = \mathcal{L} \sqcup \mathcal{P}_2 \sqcup \mathcal{P}_3$.

\smallskip

\noindent (1) First note that $L^2=\F e_3$ is in every maximal subalgebra $M$. If $e_4\notin M$ then $L=M+\F e_4$, so $e_1+\lambda e_4, e_2+\mu e_4\in M$ and $M$ is as in $\mathcal{P}_{3,1}$. If $e_4\in M$, then $M$ is as in $\mathcal{P}_{3,2}(\lambda,\mu)$.

\smallskip 

\noindent (2) Let $U$ be a two-dimensional subalgebra of $L$. It is clear from (1) that $U$ is not maximal in $L$, so $U$ is a proper subalgebra of $\mathcal{P}_{3,1}(\lambda,\mu)$ or of  $\mathcal{P}_{3,2}(\lambda,\mu)$. If $U$ belongs to the former, then $e_3\in U$, since, otherwise $e_1+\lambda e_4+\alpha e_3, b+\mu e_4+\beta e_3\in U$, which yields $e_3\in U$, a contradiction. Thus, $U$ is of the form $\F e_3+\F(\alpha e_1+\beta e_2+\gamma e_4)$. Each of these is spanned by $e_3$ and a line in a three-dimensional space, so, if $\F=\F_q$, there are $q^2+q+1$ of them. 
\par

The latter are abelian, so every two-dimensional subspace of each one is a subalgebra. So, when $\F=\F_q$, each one has $q^2+q+1$ subalgebras, one of which is $Z$ in each case, so there are $q^2+q$ which are not equal to $Z$. Now there are $q+1$ subalgebras of the form $\mathcal{P}_{3,2}(\lambda,\mu)$, so the toal number of such subalgebras is $(q+1)(q^2+q)+1$, the final one being $Z$.

\smallskip

\noindent (3) This follows from Lemma \ref{g-adjacency} (1) again.

\smallskip

\noindent (4) This follows from Lemma \ref{g-adjacency} (4), since the maximal subalgebras are the planes in $\mathcal{P}_3$.

\smallskip

\noindent (5) This follows from Lemma \ref{g-adjacency} (5), since the maximal subalgebras are the planes in $\mathcal{P}_3$.
    
\end{proof}
  \textbf{Case 3. $\dim L'=2$ - The filiform algebra}
This algebra has basis $e_1,e_2,e_3,e_4$ with $[e_1,e_2]=e_3$ and $[e_1,e_3]=e_4$. 

\begin{proposition}\label{fil} 
The proper nonzero subalgebras of $L$ are as given below.
 \begin{enumerate}
\item The three-dimensional subalgebras are of the form $\mathcal{P}_3(\lambda,\mu)=\F e_3+\F e_4 +\F(\lambda e_1+\mu e_2)$. If $\lambda \neq 0$ these are Heisenberg algebras; if $\lambda = 0$ this is abelian. When $\F=\F_q$, there are $q+1$ of these in total.
\item The two-dimensional subalgebras are the two-dimensional subspaces of $\F e_2+\F e_3+\F e_4$ together with $\F (e_1+\alpha e_2+\beta e_3)+\F e_4$. When $\F=\F_q$,there are $2q^2+q+1$ of these.
\item  If $A, B \in \mathcal{P}_3$ with $A \neq B$, then $A \sim B$. So, if $\F=\F_q$, $\Gamma(L)[\mathcal{P}_3]\cong K_{q^2q+1}$; that is, the induced subgraph on $\mathcal{P}$ is the complete graph $K_{q^2+q+1}$.
 \item If $A,B\in \mathcal{P}_2$, the set of two-dimensional subalgebras, then $A\sim B$ if and only if they don't belong to the same plane in $\mathcal{P}_3$.

\item If $A\in \mathcal{L}$ and $B\in \mathcal{P}_2$ then $A\sim B$ if and only if they don't belong to the same plane in $\mathcal{P}_3$.
 \end{enumerate}   
\end{proposition}
\begin{proof}
The comaximal graph $\Gamma(L)$ has vertex set $\mathcal{V} = \mathcal{L} \sqcup \mathcal{P}_2 \sqcup \mathcal{P}_3$.

\smallskip

\noindent (1) This follows because every maximal subalgebra must contain $L^2=\F e_3+\F e_4$.

\smallskip

\noindent (2) Note that $\mathcal{P}_3(0,\mu)$ is abelian, so every two-dimensional subspace of it is a subalgebra. Over $\F_q$ there are $q^2+q+1$ of these. If $\lambda \neq 0$ then $Fe_4$ must belong to any two-dimensional subalgebra, so the subalgebras are of the form $\F (e_1+\alpha e_2+\beta e_3)+\F e_4$ and, over $\F_q$, there are $q^2$ of these. So there are $2q^2+q+1$ in total.

\smallskip

(3)-(5) follow as in Proposition \ref{heis+}.
\end{proof}

\subsection{Comaximal graph of $\mathfrak{gl}_2(\F)$} 

We use the following, which is proved in \cite[Proposition 4.2]{TZG}.

\begin{proposition}\label{comax1} 
Using the standard basis $h,x,y$ for $L= \mathfrak{sl}_2(\F)$, where $[h,x]=2x$, $[h,y]=-2y$, $[x,y]=h$ (characteristic of $\F\neq 2$), the two-dimensional Borel subalgebras are $B=\F x + \F h$ and $B(\alpha)=\F(h+\alpha x) + \F(y +\frac{1}{4} \alpha^2 x)$ with $\alpha \in \F$. The nonsplit semisimple lines are $L(\mu,\nu)=\F(h+\mu x+\nu y)$ where  $\nu\neq 0$ and $\mu\nu+1$ is not a square in $\F$ and $L(\beta)=\F(y +\beta x)$ where $4\beta$ is not a square in $\F$. .
\end{proposition}

\begin{proposition}
Let $L= \mathfrak{gl}_2(\F)=S\oplus Z$, where $S=sl_2(F)$ and $Z=Z(L)=Fz$. Use the standard basis $h,x,y$ for $S$, where $[h,x]=2x$, $[h,y]=-2y$, $[x,y]=h$ (characteristic of $F\neq 2$). Then
\begin{enumerate}[(i)]
\item the maximal subalgebras are $S$, $M=\F x + \F h+\F z$, $M(\alpha)=\F (h+\alpha x) + \F (y +\frac{1}{4} \alpha^2 x)+\F z$ with $\alpha \in \F$, $\F z+\F (h+\mu x+\nu y)$ where  $\nu\neq 0$ and $\mu\nu+1$ is not a square in $\F$, and $F z+\F (y +\beta x)$ where $4\beta$ is not a square in $\F$. If $\F=\F_q$ ($q$ odd), there are $q+2+\frac{q(q-1)}{2}$ of these.
\item The two-dimensional subalgebras are $\F x+\F (\lambda h+\mu z)$, $\F z+\F (\lambda h+\mu x)$, $\F (h+\alpha x)+\F (y +\frac{1}{4} \alpha^2 x+\beta z)$ and $\F (h+\frac{\alpha}{2} x -\frac{2}{\alpha} y)+F(h+\alpha x+\mu z)$, where $\mu\not=0$. If $\F=\F_q$ ($q$ odd), then there are $2q^2+2q+1$ of these. 
\end{enumerate}
\end{proposition}

\begin{proof} 
\begin{enumerate}[(i)]
\item Let $M$ be a maximal subalgebra of $L$. If $Z\not\subseteq M$, then $L=M\oplus Z$ and $M\cong S$. But $M=M^2=L^2=S$, so there is one such subalgebra.
\par

If $Z\subseteq M$, the maximal subalgebras are spanned by $Z$ together with a Borel subalgebra, or with a non-split semisimple line. Using Proposition \ref{comax1} this gives the subalgebras listed. In $sl_2(\F_q)$ ($q$ odd) there are $q+1$ Borel subalgebras and $\frac{q(q-1)}{2}$ non-split semimple lines (see \cite[Proposition 4.10]{TZG}).

\item  The remaining two-dimensional subalgebras will be inside $S$, $M$ or $M(\alpha)$. Those inside $S$ will be $B$ and $B(\alpha)$.
\par

Let $U$ be a two-dimensional subalgebra of $L$ with $U\subset M$. If $x\in U$, then $U=\F x+\F (\lambda h+\mu z)$. Note that $B$ is one of these. If $x\notin U$ then $M=\F x+U$. Hence $h+\alpha x\in U$ and $z+\beta x\in U$, so $2\beta x=[h+\alpha x,z+\beta x]\in U$. It follows that $\beta=0$ and $U=\F z+\F (\lambda h+\mu x)$.
\par

Suppose now that $U$ is a two-dimensional subalgebra of $L$ with $U\subset M(\alpha)$. If $h+\alpha x\in U$. Then $U=\F (h+\alpha x)+\F (y +\frac{1}{4} \alpha^2 x+\beta z)$. Note that $B(\alpha)$ is one of these. If $h+\alpha x\notin U$, then $M(\alpha)=\F (h+\alpha x)+U$. Hence $h+\alpha x+\lambda(y+\frac{1}{4} \alpha^2 x)\in U$ and $h+\alpha x+\mu z\in U$, where $\lambda,\mu\not=0$. Then $[h+\alpha x+\lambda(y+\frac{1}{4} \alpha^2 x),h+\alpha x+\mu z]=2\lambda y-\frac{1}{2}\lambda\alpha^2 x-\lambda\alpha h\in U$ if and only if $2\lambda y-\frac{1}{2}\lambda\alpha^2 x-\lambda\alpha h=A(h+\alpha x+\lambda(y+\frac{1}{4} \alpha^2 x))+B(h+\alpha x+\mu z)$. Comparing coefficients, we have $B=0$, $A=-\lambda\alpha$ and $A\lambda=2\lambda$, so $A=2$ and $\lambda=-\frac{2}{\alpha}$. It follows that $U=\F (h+\frac{\alpha}{2} x -\frac{2}{\alpha} y)+F(h+\alpha x+\mu z)$, where $\mu\not=0$.
\par

Now suppose that $\F=\F_q$ ($q$ odd). The subalgebras containing $Z$ are $Z$ together with a line in $sl_2(\F_q)$ and there are $q^2+q+1$ of these. The remainder are inside a subalgebra containing a Borel subalgebra and $Z$, of which there are $q+1$, and, within each of these, there are $q$ subalgebras that do not contain $Z$. Hence, in total, there are $$q^2+q+1+q(q+1)=2q^2+2q+1$$ two-dimensional subalgebras.
\end{enumerate}
\end{proof}

\section{Relationships between $\Gamma(L)$ and properties of the Lie algebra $L$}
The subalgebras that are inside $F(L)$ are the isolated points of the graph, by \cite[Lemma 3.2]{TZG}. Hence, if there are no isolated points, then $F(L)=0$. Now suppose that $F(L)=\phi(L)$, the Frattini ideal (so characteristic zero or solvable $L$). If we remove the isolated points, then what we are left with represents all the subalgebras $A$ such that $(A+\phi(L))/\phi(L)$ is a proper nonzero subalgebra of $L/\phi(L)$.

\begin{lemma} \label{5.1}
If $A$ is a proper subalgebra of $L$ such that $A\not\subseteq \phi(L)$ then $(A+\phi(L))/\phi(L)$ is a proper nonzero subalgebra of $L/\phi(L)$. Moreover, if $U/\phi(L)$ is a proper nonzero subalgebra of $L/\phi(L)$, then there is a proper subalgebra $A$ such that $A\not\subseteq \phi(L)$ and $U=\phi(L)+A$.
\end{lemma}
\begin{proof}  If $A$ is a subalgebra of $L$ such that $A\not\subseteq \phi(L)$ then $(A+\phi(L))/\phi(L)$ is clearly a subalgebra of $L/\phi(L)$. It is nonzero since $A\not\subseteq \phi(L)$, and it is proper since $A+\phi(L)=L$ implies that $A=L$.
\par

Now let $U/\phi(L)$ be a proper nonzero subalgebra of $L/\phi(L)$ and write $U=\phi(L)\dot{+} V$ where $V$ is a complementary subspace to $\phi(L)$ in $U$. Put $A=\langle V\rangle$. Then $A\subset U$ and $U=\phi(L)+A$.
\end{proof}
\medskip

However, note that $A+\phi(L)=B+\phi(L)$ does not necessarily imply that $A=B$. 
\par

It is clear that if two Lie algebras have the same subalgebra lattice then they will have the same graph $\Gamma(L)$. This means that the graph cannot distinguish between an abelian Lie algebra and an almost abelian one, nor between a nilpotent algebra and an almost nilpotent one (see \cite{tow}). However, it does not give complete information about the subalgebra lattice. For example if $\dim F(L)>2$ then, because the subalgebras inside $F(L)$ are isolated points, it gives no information about inclusions between them. The graph essentially tells us whether or not two subalgebras belong to the same maximal subalgebra. 
\par

We have the following two results from \cite{TZG}.

\begin{theorem}(\cite[Theorem 3.3]{TZG})
$\Gamma(L)$ is complete if and only if every proper subalgebra is one-dimensional.
\end{theorem}

\begin{theorem}\label{p:gen} (\cite[Theorem 3.4]{TZG})
Let $L$ be a Lie algebra over a perfect field $\mathbb{F}$ of characteristic zero or $p>3$. Then every proper subalgebra of $L$ is one-dimensional if and only if either
\begin{itemize}
\item[(i)] $\dim L\leq 2$, or
\item[(ii)] $L$ is three-dimensional simple and $\sqrt{\mathbb{F}} \not \subseteq \mathbb{F}$. 
\end{itemize}
\end{theorem}
\medskip

If the underlying field is finite and we know the dimension of $L$ then we can deduce more by counting the number of vertices in $\Gamma(L)$.

\begin{lemma}\label{abalab}
If $\dim L=n$ and $\F=\F_q$, then the number of vertices in $\Gamma(L)$ is $\sum_{k=1}^{n-1}\binom{n}{k}_q$, where $\binom{n}{k}_q$ is the Gaussian binomial coefficient, if and only if $L$ is abelian or almost abelian.
\end{lemma}
\begin{proof}
This follows because $\binom{n}{k}_q$ is the number of proper nonzero subspaces in $L$, and so every such subspace is a subalgebra.
\end{proof}

\begin{proposition} If $\dim (L/\phi(L)=n$, $\F=\F_q$, and the number of vertices in $\Gamma(L/\phi(L))$ is $\sum_{k=1}^{n-1}\binom{n}{k}_q$ then $L$ is supersolvable. Moreover, either $L$ is nilpotent, or $L^2=N(L)$, the nilradical of $L$, has codimension one in $L$.  
\end{proposition}
\begin{proof} By Lemma \ref{abalab}, $L/\phi(L)$ is abelian or almost abelian. Hence, $L/\phi(L)$ is supersolvable, and so $L$ is supersolvable, by \cite[Theorem 6]{barnes}.
\par

If $L/\phi(L)$ is abelian, then $L$ is nilpotent. If $L/\phi(L)$ is almost abelian, then $L^2/\phi(L)=(L/\phi(L))^2$ has codimension one in $L/\phi(L)$. Also $L^2\subseteq N(L)$ so $L^2=N(L)$.
\end{proof}

\begin{lemma}
If $\dim L=3$ and $\F=\F_q$ with $q>3$, then the number of vertices in $\Gamma(L)$ is $q^2+q+1$ if and only if it is a three-dimensional non-split simple Lie algebra.
\end{lemma}
\begin{proof}
The number of vertices in $\Gamma(L)=q^2+q+1$ if and only if the only subalgebras are one-dimensional. The result therefore follows from Theorem \ref{p:gen}.   
\end{proof}

\section{The comaximal graph and the factor algebra \text{$L/F(L)$}}




Throughout this section we assumne that $F(L)=\phi(L)$. We recall the following standard graph-theoretic notion, which will be used to describe the relationship between $\Gamma_F(L)$ and $\Gamma(L/F(L))$.

\begin{definition}\label{blowup0}
Let $\Gamma = (V, E)$ be a simple undirected graph. A \emph{blow-up} of
$\Gamma$ is a graph $\Gamma' = (V', E')$ together with a surjective map
$\psi : V' \to V$ such that:
\begin{enumerate}
\item[(a)] for each $v \in V$, the fiber $\psi^{-1}(v)$ is a nonempty independent
      set in $\Gamma'$;
\item[(b)] for any $u, v \in V$ with $u \neq v$ and any $x \in \psi^{-1}(u)$,
      $y \in \psi^{-1}(v)$,
      \[x \sim y \text{ in } \Gamma'   \iff  u \sim v \text{ in } \Gamma.\]
\end{enumerate}
The integers $\{|\psi^{-1}(v)|\}_{v \in V}$ are called the \emph{fiber sizes}
of the blow-up. If all fiber sizes equal a common value $m \geq 1$, the
blow-up is called \emph{uniform} with \emph{multiplicity} $m$. When $m = 1$,
condition (2) reduces to $\psi$ being a graph isomorphism; thus, every graph is a uniform blow-up of itself with multiplicity $1$.
\end{definition}

\begin{theorem}\label{blowup}
Let $L$ be a finite-dimensional Lie algebra with $F(L) \neq L$, and let
$\pi : L \to L/F(L)$ be the canonical projection. For every proper subalgebra
$A$ of $L$ with $A \not\subseteq F(L)$, the image $\pi(A)$ is a proper
nonzero subalgebra of $L/F(L)$, and for all such $A, B$,
\[A \sim B \text{ in } \Gamma(L)\iff \pi(A) \sim \pi(B) \text{ in } \Gamma(L/F(L)).\]
In particular:
\begin{enumerate}
\item If $F(L) = 0$, then $\pi$ induces a graph isomorphism $\Gamma(L) \cong \Gamma(L/F(L))$.

\item If $F(L) \neq 0$, then $\pi$ is not injective on subalgebras, and
      $\Gamma_F(L)$ is a non-trivial blow-up of $\Gamma(L/F(L))$ in the
      sense of Definition \ref{blowup0}, with surjection
      $\psi : V(\Gamma_F(L)) \to V(\Gamma(L/F(L)))$ given by $\psi(A) =
      \pi(A)$, and with each fiber $\psi^{-1}(\pi(A))$ being an independent
      set of size at least $2$. In this case $\Gamma_F(L)$ is a supergraph
      of $\Gamma(L)^*$, with equality if and only if $\Gamma(L/F(L))$ has
      no isolated vertices.
\end{enumerate}
\end{theorem}

\begin{proof}
By Lemma \ref{5.1}, $\pi(A) = (A+F(L))/F(L)$ is a proper
subalgebra of $L/F(L)$; it is nonzero since $A \not\subseteq F(L)$. Since
$\pi$ is a surjective Lie algebra homomorphism,
\[\pi(\langle A, B \rangle) = \langle \pi(A), \pi(B) \rangle,\]
so $\langle \pi(A), \pi(B) \rangle = L/F(L)$ if and only if $\langle A, B
\rangle + F(L) = L$. By Lemma \ref{5.1}, the latter forces
$\langle A, B \rangle = L$, and the converse is immediate. This establishes
the adjacency equivalence.

In particular:

\noindent (1) If $F(L) = 0$ then $\pi$ is the identity and $\pi(A) = A$ for all $A$, so $\pi$ is bijective; together with the adjacency equivalence, it is a graph isomorphism.

\noindent (2) If $F(L) \neq 0$, two distinct subalgebras $A \neq B$ with $\pi(A) = \pi(B)$ show that $\psi = \pi|_{V(\Gamma_F(L))}$ is not injective. We verify the two conditions of Definition \ref{blowup0}:
\begin{enumerate}[(a)]
\item If $\psi(A) = \psi(B)$, then $A \sim B$ in $\Gamma_F(L)$ would require
$\pi(A) \sim \pi(A)$ in $\Gamma(L/F(L))$, which is impossible in a simple graph. Hence, each fiber is an independent set.

\item For $A \in \psi^{-1}(u)$ and $B \in \psi^{-1}(v)$ with $u \neq v$, adjacency $A \sim B$ holds if and only if $u = \pi(A) \sim \pi(B) = v$, by the adjacency equivalence, independently of the choice of representatives within the fibers.
\end{enumerate}
Hence $\Gamma_F(L)$ is a blow-up of $\Gamma(L/F(L))$ via $\psi$. Each fiber
has size at least $2$ since $F(L) \neq 0$ implies that distinct subalgebras
$A$ and $A + F(L)$ (for any $A \not\subseteq F(L)$) map to the same image
under $\pi$. For the final assertion: by \cite[Lemma 3.2]{TZG} every
vertex $A \subseteq F(L)$ is isolated in $\Gamma(L)$, so $\Gamma(L)^*
\subseteq \Gamma_F(L)$. Equality holds if and only if no vertex of
$\Gamma_F(L)$ is isolated, which by the adjacency equivalence occurs if and
only if $\Gamma(L/F(L))$ has no isolated vertices.
\end{proof}

Let $T(\Gamma)$ denote the set of triangles of a graph $\Gamma$. 

\begin{corollary}\label{triangle-blowup}
Under the hypotheses and notation of Theorem \ref{blowup}(2),
\[\mathfrak{t}(\Gamma_F(L)) = \sum_{\{u,v,w\} \in T(\Gamma(L/F(L)))}
|\pi^{-1}(u)|\,|\pi^{-1}(v)|\,|\pi^{-1}(w)|.\]
In particular, if the blow-up is uniform with multiplicity $m$, then 
\[ \mathfrak{t}(\Gamma_F(L)) = m^3 \, \mathfrak{t}(\Gamma(L/F(L))).\]
\end{corollary}

\begin{proof}
Since each fiber of $\pi$ is an independent set in $\Gamma_F(L)$, no triangle
of $\Gamma_F(L)$ contains two vertices from the same fiber. Hence every
triangle of $\Gamma_F(L)$ corresponds uniquely to a triangle $\{u,v,w\}$ of
$\Gamma(L/F(L))$ together with an independent choice of one vertex from each
of $\pi^{-1}(u)$, $\pi^{-1}(v)$, $\pi^{-1}(w)$; conversely, every such
choice yields a triangle of $\Gamma_F(L)$ by Theorem \ref{blowup}.
Counting these choices gives the stated formula.
\end{proof}

We now determine $F(L)$ in all eight cases of Section 3, establishing which
of the two parts of Theorem \ref{blowup} applies in each.

\begin{lemma}\label{frattini-zero}
In each of the cases of Propositions \ref{Abelian}, \ref{B2}, \ref{irreducible}, \ref{2EigenVal}, \ref{CaseC4}, and \ref{CaseD}, we have $F(L) = 0$.
\end{lemma}

\begin{proof}
In each case, $F(L) = 0$ follows by identifying the maximal subalgebras and
observing that their intersection is trivial: in Propositions \ref{Abelian} and \ref{CaseC4}, the $q^2+q+1$ planes are all maximal, and no line is contained in all of them; in Propositions \ref{irreducible} and \ref{CaseD}, the lines outside $L'$ (respectively, the non-split lines) are themselves maximal subalgebras, and
any two of them intersect trivially; in Proposition \ref{2EigenVal}, the
intersection of all planes through $E_1$ is $E_1$ and of all planes through
$E_2$ is $E_2$, so $F(L) \subseteq E_1 \cap E_2 = 0$; and in Proposition \ref{B2}, the planes satisfy $\langle y \rangle \cap \langle z \rangle = 0$ as shown in the proof of Proposition \ref{B2}.
\end{proof}

\begin{corollary}\label{frattini-classification}
Among the eight cases of Propositions \ref{Abelian} to \ref{CaseD},
the Frattini subalgebra $F(L)$ is nonzero precisely for the Heisenberg algebra
(Proposition \ref{Heisenberg}) and the algebra of Proposition \ref{1EigenVal}, where in both cases $\dim F(L) = 1$; in the remaining six cases $F(L) = 0$ by Lemma \ref{frattini-zero}. 
Correspondingly, $\Gamma(L)$
has isolated vertices exactly in these two cases, and precisely those
contained in $F(L)$: in the six cases with $F(L) = 0$, direct inspection of
Propositions \ref{Abelian}, \ref{B2}, \ref{irreducible}, \ref{2EigenVal}, \ref{CaseC4}, and \ref{CaseD} shows that
$\Gamma(L)$ has no isolated vertices at all, so $\Gamma(L)^* = \Gamma(L)$.

In the two cases with $F(L) \neq 0$, the canonical projection $\pi : L \to
L/F(L)$ exhibits $\Gamma_F(L) = \Gamma(L)^*$ as a uniform blow-up of
$\Gamma(L/F(L)) \cong K_{q+1}$ with multiplicity $m = q+1$.
Corollary \ref{triangle-blowup} then yields
\[\mathfrak{t}(\Gamma_F(L)) =  (q+1)^3\binom{q+1}{3} = \frac{q(q-1)(q+1)^4}{6},\]
recovering Propositions \ref{Heisenberg} and \ref{1EigenVal}  by a
single uniform computation, without the case-by-case analysis of Section 3.
\end{corollary}

\begin{example}\label{blowup-recovers}
We illustrate the blow-up structure for the Heisenberg algebra
(Proposition \ref{Heisenberg}). Here $F(L) = Z(L) = \langle h \rangle$
has dimension $1$, and $L/F(L)$ is the two-dimensional abelian algebra, so
$\Gamma(L/F(L)) \cong K_{q+1}$. The fibers of $\pi$ are computed explicitly:
the fiber over the image of each plane $P \in \mathcal P$ consists of $P$
itself together with the $q$ non-central lines contained in $P$, giving
$|\pi^{-1}(v)| = q+1$ for each $v \in V(\Gamma(L/F(L)))$. Since
$\Gamma(L/F(L)) \cong K_{q+1}$ has no isolated vertices, $\Gamma(L)^* =
\Gamma_F(L)$. The same analysis applies to Proposition \ref{1EigenVal}:
the fiber over $V/E$ is $\{V\} \cup \{A_1,\ldots,A_q\}$, and the fiber over
each $Q_{b_0}/E$ is $\{Q_{b_0}\} \cup \{X_{a,b_0} : a \in \mathbb F_q\}$;
again every fiber has size $q+1$, confirming the uniform blow-up structure
with multiplicity $m = q+1$ in both cases.
\end{example}

{\bf Remark} We finish by answering a question raised in relation to \cite[Proposition 3.5]{TZG}. It asks whether the diameter of the induced graph $\Gamma_F(L)$ is at most two when $L$ is a Lie algebra over an algebraically closed field $\F$. We show that this is not the case by means of a final example.

\begin{example}
Let $L= Fz\oplus S$ where $Fz$ is the centre of $L$ and $S\cong sl_2(\F)$. Then $F(L)=0$, and we claim that there is no path of length two from $\F z$ to $\F s$ where $s\in S$. For, suppose there is a subalgebra $C$ such that $\F z\sim C\sim \F s$. Then $L=\langle\F z,C\rangle=\F z+C$, so  $S=L^2=C^2\subseteq C$, whence $C=S$. But $\langle S,\F s\rangle=S\not = L$.
\end{example}

\bibliographystyle{plane}

\end{document}